\documentclass[a4paper,10pt]{article}
\usepackage[utf8]{inputenc}

\usepackage{authblk}
\usepackage{fullpage}
\usepackage{amsmath,amsthm,amssymb}
\usepackage{graphicx}

\usepackage{enumerate}

\newcommand{\C}{\mathbb{C}}
\newcommand{\R}{\mathbb{R}}
\newcommand{\N}{\mathbb{N}}
\newcommand{\E}{\mathbb{E}}
\newcommand{\PP}{\mathbb{P}}
\newcommand{\la}{\lambda}
\newcommand{\1}{\mathbf{1}}
\newcommand{\ind}{1\!\!1}

\newtheorem{dfn}{Definition}

\newtheorem{thm}{Theorem}
\newtheorem{rem}{Remark}

\title{Measure on gauge invariant symmetric norms}
\author{A. Lovas}
\author{A. Andai}
\affil{Department of Analysis, Budapest University of Technology and Economics}

\begin{document}

\maketitle

\begin{abstract}
The concept of a gauge invariant symmetric random norm
is elaborated in this paper. We introduce norm processes
and show that this kind of stochastic processes are
closely related to
gauge invariant symmetric random norms. We construct a
gauge invariant symmetric random norm on the plain.
We define two different extensions of these random norms
to higher (even infinite) dimensions. We calculate
numerically unit spheres of expected norms in two and
three dimensions for the constructed random norm.  
\end{abstract}

% \tableofcontents

\bigskip
\section{Introduction} % Bevezetés

\medskip
We do not need to emphasize that norms and metrics induced by norms play
very important role in analysis. 
A norm $||.||:\C^n\to [0,\infty)$ on $\C^n$ is called \emph{gauge invariant} 
if it satisfies the condition
\[
\forall x\in\C^n\quad ||x|| = ||\,|x|\,||
\]
where $|.|$ denotes the element-wise absolute value of the vector, and it is
said to be \emph{symmetryc} if it satisfies the condition
\[
\forall \pi\in S_n\,\,\forall x\in\C^n \quad ||x\circ\pi|| = ||x|| 
\]
where $S_n$ denotes the symmetric group of order $n$.
Obviously, familiar $p$-norms frequently used in analysis possess these properties.
Note that gauge invariant symmetric norms are determined by those on 
$\R_{+,\ge}^n:=\{(x_1,\ldots,x_n)| x_1\ge x_2\ge \ldots \ge x_n \ge 0\}$
\cite{bhatia}. All kind of norms considered in this paper
has the property that the norm of the vector $(1,0,0,\ldots,0)$ is
equal to $1$ which is called the normalization convention.

There are some applications when it would be useful
to define a measure on gauge invariant symmetric norms which 
can be given on the considered vector space. 
These measures can be interpreted as a random norm.

\medskip
\begin{dfn}\label{rn}
 Let $(\Omega,\mathcal{F},\PP)$ be a probability space. 
 A function $p:\Omega\times\C^n\to [0,\infty)$ is called
 \emph{gauge invariant symmetric random norm} (GSRN) if the following
 conditions hold:
 \begin{enumerate}[(i)]
  \item 
  $\PP\left(\left\lbrace\omega\in\Omega\left|
  \,
  p(\omega,.):\C^n\to [0,\infty) 
  \text{ is a symmetric gauge invariant norm.}
  \right.\right\rbrace\right)=1$.
  
  % $\forall \omega\in\Omega\quad p(\omega,.):\C^n\to [0,\infty)$ is a symmetric
  % gauge invariant norm.
  \item $\forall x\in\C^n\quad p(.,x):\Omega\to [0,\infty)$ is a random variable.
 \end{enumerate}
\end{dfn}
\noindent
It can be easily deduced from the definition of gauge invariant symmetric norm 
and the normalization convention that
\begin{equation}\label{inh}
 \forall x\in\C^n\quad
 \PP\left(\left\lbrace
 \omega\in\Omega
 \left|
 p(\omega,x)\notin [||x||_{\infty},||x||_1]
 \right.
 \right\rbrace\right) = 0
\end{equation}
holds, and for any $x\in\C^n$ the expectation of $p(.,x)$ is a symmetric gauge invariant norm.

\smallskip
The paper is organized as follows. 
Section 2 is divided into three subsections.
In subsection 2.1 and 2.2, we introduce norm processes and
theirs path integral representation.
In subsection 2.3, we give some basic definitions
for continuous-time Markov chains and we define 
Markovian GSRNs.
In section 3, we present an efficient tool 
for studying the distribution of continuous-time 
Markov chain time integrals and we
construct a Markovian GSRN on the plain.
Section 4 deals with higher dimensional generalizations
of GSRNs and an open problem is presented in this section.
Finally, in section 5 an application is presented.

\bigskip
\section{GSRN on the plain} 
% Konstrukció R^2-ben

\bigskip
\subsection{Norm processes}

\medskip
Obviously, a symmetric gauge invariant norm restricted to
$\R^2_{+,\ge}$ can be given by its unit circle which can be
parametrized by the second coordinates of points lying on it.
This observation motivates the following definition.
\begin{dfn}\label{nf}
 A real valued stochastic process $(X_t)_{t\ge 0}$ 
 is called a \emph{norm process} if its
 realizations satisfy the following conditions 
 $\PP$-a.s.:
 \begin{enumerate}[(i)]
  \item $X_0 = 0$,
  \item $\forall\, 0\le t_1<t_2 \quad 0\le \frac{X_{t_2}-X_{t_1}}{t_2-t_1}\le 1$ and
  \item $t\mapsto X_t$ is convex and continuous.
 \end{enumerate}
\end{dfn} 

The next Theorem states that norm processes can be considered as
a paramerization of the unit circles restricted to $\R^2_{+,\ge}$.
\begin{thm}
Let $(X_t)_{t\ge 0}$  be a norm process, the corresponding
probability space of which is $(\Omega,\mathcal{F},\PP)$.
It is true
for $\PP$-a.s. $\omega\in\Omega$ that
for any $v=(v_1,v_2)\in\R^2_{+,\ge}\setminus \{(0,0)\}$ 
there is a unique $p\in  \left[||v||_\infty,||v||_1 \right]$
such that
\[
\frac{v_1}{p}+X_{\frac{v_2}{p}}(\omega) = 1
\]
holds and the function 
$p:\Omega\times\R^2_{+,\ge}\setminus \{(0,0)\}\to [0,\infty)$,
which is defined for $\PP$-a.s. $\omega\in\Omega$, 
extended to $0\in\R^2_{+,\ge}$ as $p(0):=0$ is a GSRN.
\end{thm}
\begin{proof}
 
 Suppose that for $\omega\in\Omega$ conditions $i.-iii.$
 in definition \ref{nf}. are satisfied. Let 
 $v=(v_1,v_2)\in\R^2_{+,\ge}\setminus \{(0,0)\}$
 be an arbitrary vector. The function
 $(0,\infty)\ni p\mapsto \frac{v_1}{p}+X_{\frac{v_2}{p}}(\omega)$
 is continuous, strictly decreasing and
\begin{equation}
 \begin{split}
  \frac{v_1}{v_1}+X_{\frac{v_2}{v_1}}        (\omega) &\ge 1\\
  \frac{v_1}{v_1+v_2}+X_{\frac{v_2}{v_1+v_2}}(\omega) &\le 1
 \end{split}
\end{equation}
hold because $0\le X_t\le t$.
This implies that there exists a unique 
$p\in  \left[||v||_\infty,||v||_1 \right]$ 
for which 
%\[
$
\frac{v_1}{p}+X_{\frac{v_2}{p}}(\omega) = 1
$.
%\]
%is held.

Let us consider the extension of $p$ and choose an
$\omega\in\Omega$ as above.
\begin{enumerate}[1.]
 \item For all $v\in\R^2_{+,\ge}$
 $p(\omega,v)=0 \Leftrightarrow v=0$
 because $p(\omega,v)\in\left[||v||_\infty,||v||_1 \right]$.
 
 \item For all $\alpha >0$ $\frac{\alpha v_1}{p(\omega,\alpha v)} + X_{\frac{\alpha v_2}{p(\omega,\alpha v)}}=1$
 hence $p(\omega,\alpha v) = \alpha p(\omega,v)$.
 
 \item If $v,w\in\R^2_{+,\ge}$ are nonzeros vectors, then 
 due to the convexity of $t\mapsto X_t(\omega)$ we have
 \begin{equation}
 \begin{split}
 1&=\frac{v_1+w_1}{p(\omega,v)+p(\omega,w)}
 +
 \frac{p(\omega,v)}{p(\omega,v)+p(\omega,w)}
 X_{\frac{v_2}{p(\omega,v)}}
 +
 \frac{p(\omega,w)}{p(\omega,v)+p(\omega,w)}
 X_{\frac{w_2}{p(\omega,w)}}\\
 &\ge \frac{v_1+w_1}{p(\omega,v)+p(\omega,w)} +
 X_{\frac{v_2+w_2}{p(\omega,v)+p(\omega,w)}}
 \nonumber
 \end{split}
 \end{equation}
 hence $p(\omega,v+w)\le p(\omega,v)+p(\omega,w)$.
\end{enumerate}
So we have deduced that if $\omega\in\Omega$
satisfies conditions $(i)-(iii)$ in definition \ref{nf},
$p(\omega,.):\R^2_{+,\ge} \to [0,\infty)$
defines a symmetric gauge invariant norm.

Let $v\in\R^2_{+,\ge}$ be an arbitrary vector and $y\in (0,\infty)$
we have
\[
\PP \left(p(.,v)<y\right) = 
\PP \left(\frac{v_1}{y}+X_{\frac{v_2}{y}}<1\right)
=
\PP \left(X_{\frac{v_2}{y}}<1-\frac{v_1}{y}\right)
\]
which means that $p(.,v):\Omega\to \left[||v||_\infty,||v||_1 \right]$
is a random variable.
\end{proof}

\medskip
If $y\in \left[||v||_\infty,||v||_1 \right]$, then
$0\le \frac{v_2}{y}\le\frac{v_2}{v_1+v_2}\le 1$ and
$0\le 1-\frac{v_1}{y}\le 1-\frac{v_1}{v_1+v_2}\le 1$.
Consequently, it is enough to consider the function
\[
 (t,x)\mapsto \PP \left(X_t(.)<x\right)
\]
on $[0,1]^2$. Conversely, unit circles of a GSRN 
restrcited to $\R^2_{+,\ge}$ can be considered as
graphs of a pathwise restricted norm process which
cannot be identified uniquely by the GSRN.

\bigskip
\subsection{Representation of norm processes}

\medskip
We have seen that norm processes are closely related to GSRNs
therefore it would be desirable to find good representations
of it.
We know from the theory of integration that a continuous
monotone function is almost all differentiable
and it is equal to the integral function of its
almost all existing derivative \cite{realanal}.
If we apply this fact to the trajectories of a norm process
$(X_t)_{t\ge 0}$, we get that there exists a
process $\left(Z_t\right)_{t\ge 0}$ such that
\[X_t(.)\stackrel{\PP\text{-a.s.}}{=}
\int\limits_0^t Z_s(.)\,\mathrm{d}s
\]
and trivially the realizations of $\left(Z_t\right)_{t\ge 0}$
are non negative, increasing and bounded functions
whose upper bound is one. Therefore, we may assume that
\[\left(Z_t\right)_{t\ge 0} = (\tilde{F}\circ Y_t)_{t\ge 0}\]
where $(Y_t)_{t\ge 0}$ is a $\PP$-a.s. increasing 
stochastic process
in a partially ordered metric space $(S,\le)$ and 
$\tilde{F}:S\to [0,1]$ is a monotone increasing function.

\begin{equation}\label{eq:rep}
 X_t (.)\stackrel{\PP\text{-a.s.}}{=} \int\limits_{0}^t
 \tilde{F}\circ Y_s(.)\,\mathrm{d}s
\end{equation}

 The above introduced path integral representation
 suggests another representation for norm processes.
 If we assume that $\tilde{F}$ is a probability 
 distribution function
 corresponding to a random variable $\xi\in S$
 which is defined
 on a different probability space 
 $(\Lambda,\mathcal{G},\tilde{\PP})$
 and we consider $\xi$ and the process
 $(Y_t)_{t\ge 0}$
 as random processes on
 $(\Omega\times\Lambda,\mathcal{F}\otimes\mathcal{G},\PP\otimes\tilde{\PP})$
 that are independent, then we can write
 \begin{equation}
 \begin{split}
 X_t (.)\stackrel{\PP\text{-a.s.}}{=} 
 \int\limits_{0}^t \tilde{F}\circ Y_s(.)\,\mathrm{d}s
 =&\int\limits_{0}^t \tilde{\PP} (\xi < Y_s)\,\mathrm{d}s =\\
 =&
 \int\limits_{0}^t \int\limits_{\Lambda} \ind_{\xi (\eta)<Y_s(.)} \,\mathrm{d}\tilde{\PP}(\eta)\,\mathrm{d}s
 =
  \int\limits_{\Lambda} \int\limits_{0}^t \ind_{\xi (\eta)<Y_s(.)} \,\mathrm{d}s\,\mathrm{d}\tilde{\PP}(\eta)
  =\E_{\tilde{\PP}}\left((t-\tau_{\xi}(.))_+\right),
  \end{split}
 \end{equation}
where $\tau_r$ is the hitting time of level $S\ni r$: $\tau_{r} = \inf\{s\ge 0|Y_s\ge r\}$.

\bigskip
\subsection{Markovian GSRNs}
\medskip

\begin{dfn}
 A GSRN is said to be \emph{Markovian} if the associated norm process can be derived
 from a Markovian process through taking its integral in time.
\end{dfn}
\noindent
Markovian GSRNs are neither trivial nor so complicated
that we cannot understand their behaviour especially in
cases when the state space is finite. For this reason,
some elementary facts are sketched about continuous-time Markov chains.
This plays just an introductory role and more information
about continuous-time Markov chains can be found in \cite{markov}.
\begin{dfn}
 A stochastic process $(Y_t)_{t\ge 0}$ on a countable
 state space $S$ is said to be a time homogeneous
 \emph{continuous-time Markov chain}, if it is memoryless
 which means
 \[
 \PP \left(Y_t = \beta\left|Y_{t_1}=\alpha_1,\ldots Y_{t_n}=\alpha_n\right.\right)
 =
 \PP \left(Y_t = \beta\left|Y_{t_1}=\alpha_1\right.\right)
 \]
 holds for each $0< t_1,\ldots, t_n< t$ and for any 
 $\alpha_1,\ldots,\alpha_n,\beta\in S$, there exists a
 mapping $P:[0,\infty)\to \R^S\times\R^S$ for which
 \[
 \forall t\in [0,\infty)
 \quad
 \forall\alpha ,\beta\in S
 \quad
 \PP \left(Y_t = \beta\left|Y_0=\alpha\right.\right) 
 = P(t)_{\alpha \beta}
 \]
 holds and $P$ satisfies the following properties
 \begin{enumerate}[(i)]
  \item $P(0) = \mathrm{id}_{\R^S}$
  \item $\exists\lim\limits_{t\searrow 0} P(t) = \mathrm{id}_{\R^S}$
  \item $P(t+s)=P(t)P(s)$ $\forall t,s\in [0,\infty)$.
 \end{enumerate}
 For each $t\in [0,\infty)$ the matrix $P(t)$ is called \emph{transition matrix}
 corresponding to the time point $t$.
\end{dfn}
Conditions $(i)-(iii)$ in definition above imply that
there exists a unique $G\in\R^S\times\R^S$ for which
$P(t)=e^{tG}$ holds and $G$ is called the 
\emph{infinitesimal generator} of
continuous-time Markov chain \cite{markov}.
 
\bigskip 
\section{An example for GSRN}\label{exa}

\bigskip
\subsection{Path integral of continuous-time Markov chains} 
% Itt kiszámoljuk a karakterisztkus függvényt

% Ide le kellene írni, hogy ez tkp Feyman-Kac, és komplexre nem triv.

% Lie-Throtter formulás bizonyítás + megjegyzés folytonos idő esetére

\medskip
In this point path integrals of continuous-time Markov chains
are taken under consideration. 
The next Theorem is a variant of the Feynman--Kac formula \cite{oksendal}
for continuous-time Markov chains on finite state spaces.
This will enable us to compute
the distribution of Markovian GSRNs.

\begin{thm}\label{infg}
 Let $(Y_t)_{t\ge 0}$ be a continuous-time Markov chain
 on a finite state space $S$ of which the infinitesimal generator
 is $G$ 
 and let $f:S\to\R$ be an injective
 function.
 If $(X_t)_{t\ge 0}$ denotes the path integral process of
 $f\circ Y$:
 \[
 X_{t}=\int\limits_{0}^t f\circ Y_s\,\mathrm{d}s,
 \]
 then the characteristic function of $X_t$ can be expressed
 as follows
 \begin{equation}\label{semi}
 (\forall y_0\in S)
 \quad
  \E \left(e^{iuX_t}\left|Y_0=y_0\right.\right) =
e^{t(G+iuM_f)}(\1)(y_0)
 \end{equation}
 where $\1 :S\to\{1\}$ is the constant function
 and $M_f:\R^S\to\R^S$ is the operator of multiplication by $f$. 
\end{thm}

\begin{proof}
The function $t\mapsto f(Y_t(\omega))$ is a step function
for all $\omega\in\Omega$ hence its integral can be expressed
as a limit of Riemann sums. Using this and
the dominated convergence theorem
we can write the following.
 \begin{equation} 
 \begin{split}
\E \left(e^{iuX_t}\left|Y_0=y_0\right.\right) 
= \E \left(\left.e^{iu\lim\limits_{m\to\infty} \frac{t}{m}\sum\limits_{k=1}^m  f\left(Y_{\frac{kt}{m}}\right) }\right|Y_0=y_0\right) 
= \lim\limits_{m\to\infty}  \E \left(\left. e^{iu \frac{t}{m}\sum\limits_{k=1}^m f\left(Y_{\frac{kt}{m}}\right) }\right|Y_0=y_0\right)
\end{split}
 \end{equation}
 Using the injectivity of $f$ and the Markovian property of
 $(Y_t)_{t\ge 0}$ we get
 \begin{equation} 
 \begin{split}
 \E& \left(\left. e^{iu \frac{t}{m}\sum\limits_{k=1}^m f\left(Y_{\frac{kt}{m}}\right) }\right|Y_0=y_0\right) 
  = \sum_{y_1,\ldots,y_m\in S} 
 e^{iu \frac{t}{m}\sum\limits_{k=1}^m f(y_k)}
 \PP \left(\left.\bigcap\limits_{k=1}^m \left\lbrace f\left(Y_{\frac{kt}{m}}\right) = f(y_k) \right\rbrace\right|Y_0 = y_0\right)=\\
% =& \sum_{y_1,\ldots,y_m\in S} 
% e^{iu \frac{1}{m}\sum\limits_{k=1}^m f(y_k)}
% \PP \left(\left.\bigcap\limits_{k=1}^m \left\lbrace Y_{\frac{kt}{m}} = y_k \right\rbrace\right|Y_0 = y_0\right)= \\
=& \sum_{y_1,\ldots,y_m\in S} 
 \prod\limits_{k=1}^m e^{\frac{iut}{m}f(y_k)}
 \PP\left(\left.Y_{\frac{kt}{m}} = y_k\right|Y_{\frac{(k-1)t}{m}} = y_{k-1}\right).
 \end{split}
 \end{equation}
 Observe that the received expression can be written 
 as the action of the $m$-times composite of
 the operator $e^{\frac{t}{m}G}e^{\frac{iu}{m}M_f}$ on $\1$.
 \[
 \sum_{y_1,\ldots,y_m\in S} 
 \prod\limits_{k=1}^m 
 e^{\frac{iut}{m}f(y_k)}
 \PP\left(\left.Y_{\frac{kt}{m}} = y_k\right|Y_{\frac{(k-1)t}{m}} = y_{k-1}\right)
 =
 \left(e^{\frac{t}{m}G}e^{\frac{iut}{m}M_f}\right)^m(\1)(y_0)
 \]
 If we take the limit as $m\to\infty$ and
 we use the Lie–Trotter product formula
 we get the desired expression.
 \[
 \E \left(e^{iuX_t}\left|Y_0=y_0\right.\right) =
 \lim\limits_{m\to\infty}\left(e^{\frac{t}{m}G}e^{\frac{iut}{m}M_f}\right)^m 
 (\1)(y_0)
 = e^{t(G+iuM_f)}(\1)(y_0)
 \] 
\end{proof}
% Lie szorzat formula hivatkozás

\medskip
Let us introduce the following notations for the conditional characteristic
function above and the corresponding conditional distribution function.
\begin{equation}
 \begin{split}
  \varphi (t,u)&= \E \left(e^{iuX_t}\left|Y_0\right.\right)\\
  F(t,x)       &= \PP \left(X_t<x\left|Y_0\right.\right)
 \end{split}
\end{equation}
If we take the time derivative of
\eqref{semi} we get that $\varphi$ is the solution of the 
Cauchy-problem below.
\begin{equation}\label{eq8}
 \begin{split}
  \partial_1\varphi &= G\varphi + i u M_f \varphi\\
  \varphi (0,u) &=\1\in\R^S.
 \end{split}
\end{equation}
Let us introduce a normally distributed random variable $\xi$
which is independent
from $X_t$ and its variance is
$\sigma^2>0$. The characteristic function $\varphi_\sigma$ of $X_t + \xi$
satisfies equations similar to \eqref{eq8}.
\begin{equation}
 \begin{split}
  \partial_1\varphi_\sigma &= G\varphi_\sigma + M_f i u \varphi_\sigma\\
  \varphi (0,u) &=e^{-\frac{\sigma^2 u^2}{2}}\cdot\1\in\R^S
 \end{split}
\end{equation}

Assume that $\partial_1 F_\sigma (t,x)$ exists and vanishes
for all $t\in [0,\infty)$
when $x\to -\infty$.
We can write
\begin{equation}
\begin{split}
\partial_1\varphi_\sigma (t,u) 
&=\frac{\partial}{\partial t}
\int\limits_\R e^{iux}\,F_\sigma(t,\mathrm{d}x)=
\frac{\partial}{\partial t}
\int\limits_\R e^{iux}\,
\int\limits_{[0,t]}
\partial_1 F_\sigma(s,\mathrm{d}x)\,\mathrm{d}s
\\
%&\stackrel{\ast}{=}
&=
\frac{\partial}{\partial t}
\int\limits_{[0,t]}
\int\limits_\R e^{iux}\,
\partial_1 F_\sigma(s,\mathrm{d}x)\,\mathrm{d}s=
\int\limits_\R e^{iux}
\,\partial_1 F_\sigma(t,\mathrm{d}x),
\end{split}
\end{equation}
% where $\ast$ equation holds because
and
\begin{equation}
 \int\limits_\R iue^{iux}\,F_\sigma (t,\mathrm{d}x) =
 \left.-\partial_2 F_\sigma (t,x)e^{iux}\right|_{x=-\infty}^{x=\infty}+
 \int\limits_\R iue^{iux}\,F_\sigma (t,\mathrm{d}x) =
 \int\limits_\R -e^{iux} \,\partial_2 F_\sigma (t,\mathrm{d}x) 
\end{equation}
which implies that the Fourier--Stieltjes transform of the
signed Borel measure associated to
$\partial_1 F_\sigma - G F_\sigma + M_f \partial_2 F_\sigma$
is zero
\begin{equation}
\forall t\in [0,\infty)
\,
\forall u\in\R \quad
 \int\limits_\R e^{iux}
 \,
 \left[
 \partial_1 F_\sigma - G F_\sigma + M_f \partial_2 F_\sigma
 \right](t,\mathrm{d}x) = 0.
\end{equation}
On the other hand
\[
\forall t\in [0,\infty)
\quad
 \lim\limits_{x\to -\infty}
 \left[
 \partial_1 F_\sigma - G F_\sigma + M_f \partial_2 F_\sigma
 \right](t,x) = 0
\]
which implies that $F_\sigma$
is the solution of the Cauchy problem
% as follows:
\begin{equation}\label{csy}
 \begin{split}
  \partial_1 F_\sigma &= G F_\sigma - M_f \partial_2 F_\sigma\\
  F_\sigma (0,x) &= \Phi_\sigma (x) \cdot \1\in\R^S,
 \end{split}
\end{equation}
where $\Phi_\sigma$ denotes the distribution function of $\xi$.

\medskip
\begin{rem}
From the integral form of $X$ we obtain the estimation
 \[
 X_t + \xi + \Delta t\cdot m
 \le
 X_{t+\Delta t} + \xi
 \le
 X_t + \xi + \Delta t\cdot M,
 \]
 where $m=\min\limits_{s\in S} f(s)$
 and $M=\max\limits_{s\in S} f(s)$.
 Using this we get
 \[
 F_\sigma (t,x-\Delta t\cdot M)
 \le
 F_\sigma (t+\Delta t,x)
 \le
 F_\sigma (t,x-\Delta t\cdot m)
 \]
which can be applied to estimate partial derivatives of
$F_\sigma$ through estimating difference quotients
as follows:
 \[
 - M\cdot\frac{F_\sigma (t,x) - F_\sigma (t,x-\Delta t M)}{\Delta t\cdot M}
 \le
 \frac{F_\sigma (t+\Delta t,x)-F_\sigma (t,x)}{\Delta t}
 \le
  - m\cdot\frac{F_\sigma (t,x) - F_\sigma (t,x-\Delta t \cdot m)}{\Delta t\cdot m} 
 \]
 ($\le$ relation means elementwise relations).
 Consequently, upper and lower bounds were
 obtained for $\partial_1 F_\sigma (t,x)$
 \[
 -M\cdot \partial_2 F_\sigma (t,x)
 \le
 \partial_1 F_\sigma (t,x)
 \le
 -m\cdot \partial_2 F_\sigma (t,x)
 \]
 from which it follows that the assumption
 about the limit of $\partial_1 F_\sigma (t,x)$
 as $x\to -\infty$ can be omitted.
\end{rem}
The random variable $X_t + \xi$ converges weakly to $X_t$
when $\sigma\to 0$ thus we have to solve
\eqref{csy} and take the limit $\sigma\to 0$ to obtain
$F$ in all points of continuity.
The following Theorem gives an integral equation representation
for $F$.

\medskip
\begin{thm}
If $\left(X_t\right)_{t\ge 0}$ is a stochastic process defined
in Theorem \ref{infg}, then for any $y_0\in S$ the conditional
distribution $F(t,x)_{y_0}=\PP \left(X_t < x\left| Y_0 = y_0\right.\right)$
is the solution of the following integral equation
 \begin{equation}\label{eqthm}
 \begin{split}
 &F(t,x)_{y_0} =
 e^{-t\la}\ind (x\ge t\cdot f(y_0))+         \\
 &+\int\limits_0^t
 \sum\limits_{\sigma\in S\setminus \{y_0\}}
 F(s,x-(t-s)\cdot f(y_0))_{\sigma}
 \cdot
 \PP
 \left(
 Y_{t-s} =\sigma\left|Y_0 =y_0\right.
 \right)
 \la  e^{-\la (t-s)}
 \,
 \mathrm{d}s,
 \end{split}
 \end{equation}
 where the time spent by $Y_s$ in each states has exponential
 distribution with $\la$ parameter.
\end{thm}
\begin{proof}
Let us denote the time up to the first jump by
 $\tau$. By using the law of total probability
 \[
 F(t,x)_{y_0}  =
 \PP \left(X_t <   x\left| Y_0 = y_0,\tau\ge t\right.\right) 
 \PP \left(\tau\ge t\right)
 +
 \int\limits_0^t
 \PP \left(X_t < x\left| Y_0 = y_0,\tau = t-s\right.\right)
 \cdot
 \la e^{-\la (t-s)}
 \,
 \mathrm{d}s
 \]
 can be written, because $\tau$ is independent from the initial state.
 We have
 $
 \PP \left(X_t <   x\left| Y_0 = y_0,\tau\ge t\right.\right)= 
 \ind (x-t\cdot f(y_0)\ge 0)
 $.
 We get in a similar way
 \[
 \PP \left(X_t < x\left| Y_0 = y_0,\tau =t-s\right.\right)
 =
 \sum\limits_{\sigma\in S\setminus \{y_0\}}
 \PP \left(X_t < x\left| Y_0 = y_0,Y_{t-s} = \sigma,\tau = t-s\right.\right)
  \cdot
 \PP
 \left(
 Y_{t-s} =\sigma\left|Y_0 =y_0\right.
 \right)
 \]
 which is true for any $s\in [0,t)$
 and due to the Markov property of $\left(Y_t\right)_{t\ge 0}$
 \[
  \PP \left(X_t < x\left| Y_0 = y_0,Y_{t-s} = \sigma,\tau = t-s\right.\right)
  =
  F(s,x-(t-s)\cdot f(y_0))_{\sigma}
 \]
can be written which completes the proof.
\end{proof}

\bigskip
\subsection{Constuction of a GSRN} 
% Konkrét esetre a karakterisztikus függvény

\medskip
In this point we consider
$(Y_t)_{t\ge 0}$ a continuous-time
Markov chain on $S=\{0,\ldots,n\}$ ($n\in\N$) of which the
infinitesimal generator is $G=\la_n N\in\R^{(n+1)\times (n+1)}$ where
\[
N =\left(
\begin{array}{cccccc}
  -1 &  1 & 0 & \cdots & 0 & 0 \\
   0 & -1 & 1 & \cdots & 0 & 0 \\
   0 & 0 & -1 & \cdots & 0 & 0 \\
   \vdots & \vdots & \vdots & \ddots & \vdots & \vdots \\
   0 & 0 & 0 & \cdots & -1 & 1 \\
   0 & 0 & 0 & \cdots & 0 & 0
\end{array}\right)
\]
and $(\la_n)_{n\in\N}$ regulates the ``speed'' of the chain.
We will compute the distribution function
of the integral process 
\[X_t = \int\limits_0^t \frac{1}{n}Y_s\,\mathrm{d}s\]
as we presented in the previous subsection. Every piecewise
linear function on $[0,\infty )$ with $n$ line segments
can be realized as a trajectory of $\left(X_t\right)_{t\ge 0}$
thus $X_t$ defines a measure on gauge invariant symmetric norms
which unit circles are $4n$-gons. 

The Cauchy problem corresponding to the smoothed variable
can be written as
\begin{equation}
 \begin{split}
  \partial_1 (F_\sigma)_k + \frac{k}{n} \partial_2 (F_\sigma)_k
  + \la_n (F_\sigma)_k &= \la_n (F_\sigma)_{k+1}\\
  \partial_1 (F_\sigma)_n + \frac{k}{n} \partial_2 (F_\sigma)_n
  + \la_n (F_\sigma)_n &= 0\\
  (F_\sigma )_k (0,x) &= \Phi_\sigma (x),
 \end{split}
\end{equation}
where $(F_\sigma)_k = \PP (X_t + \xi < x|Y_0=k)$ and
$k= 0,\ldots, n$.
If we substitute $(F_\sigma (t,x))_k$ by $e^{-t \la_n}(J_\sigma (t,x))_k$
we obtain
a system of quasilinear partial differential equations
\begin{equation}\label{pde}
 \begin{split}
  \partial_1 (J_\sigma)_k + \frac{k}{n} \partial_2 (J_\sigma)_k
   &= \la_n (J_\sigma)_{k+1}\\   
  \partial_1 (J_\sigma)_{n} + \frac{k}{n} \partial_2 (J_\sigma)_n
   &= 0\\
  (J_\sigma )_k (0,x) &= \Phi_\sigma (x) 
 \end{split}
\end{equation}
which can be solved directly by using methods of
characteristics. We get the following recursion for $F_\sigma$
\begin{equation}
\begin{split}
 (F_\sigma)_n (t,x) &= \Phi_\sigma (x-t) \\
 (F_\sigma)_k (t,x) &= 
 e^{-t\la_n}\Phi_\sigma \left(x-\frac{k}{n}\cdot t\right) +
 \int\limits_0^t
 (F_\sigma)_{k+1} \left(t,x-\frac{k}{n}(t-s)\right)
 \la_n e^{-\la_n (t-s)}
 \,\mathrm{d}s \\
 \text{for }k & = 0,\ldots ,n-1.
\end{split}
\end{equation}
If $\sigma\to 0$ we obtain
\begin{equation}
\begin{split}
 (F)_n (t,x) &= \ind (x-t \ge 0) \\
 (F)_k (t,x) &= 
 e^{-t\la_n}\ind\left(x-\frac{k}{n}\cdot t \ge 0\right) +
 \int\limits_0^t
 (F)_{k+1} \left(t,x-\frac{k}{n}(t-s)\right)
 \la_n e^{-\la_n (t-s)}
 \,\mathrm{d}s \\
 \text{for }k & = 0,\ldots, n-1
\end{split}
\end{equation}
which is the same as we would got, if
formula \eqref{eqthm} have been used.

\medskip
System 
\eqref{pde} for $\sigma = 0$ was solved
numerically by upwind scheme \cite{upwind} and then
expectation of the GSRN,
which is a symmetric gauge invariant norm,
was computed.
In simulations the interval $[0,1]$ was divided into $N+1$
equal parts by $N$ internal points.
Graph of
$(F)_0(t,x)$
is presented in Figure \ref{fig:1},
where simulation was carried out by the following
settings:
$n=10$, $\la_n=10$ and $N=200$.
\begin{figure}[ht!]
 \centering
 \includegraphics[width=0.75\textwidth]{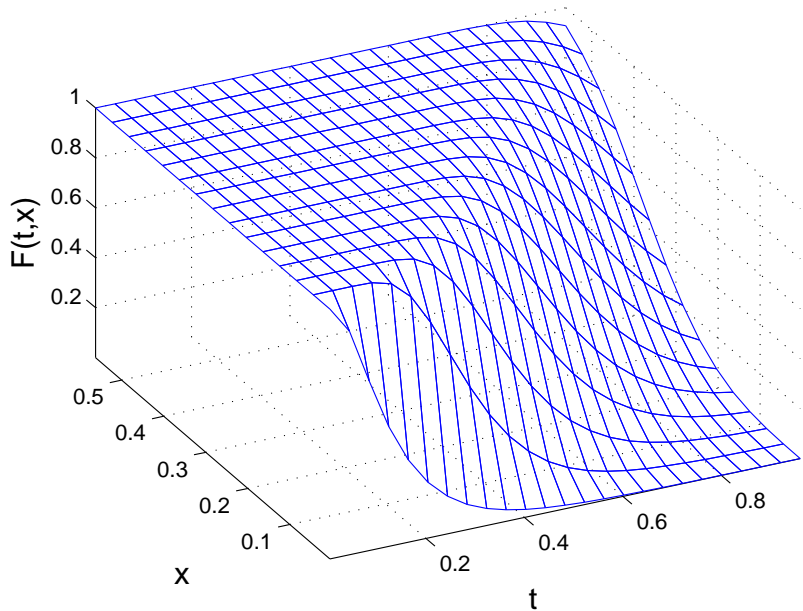}
 \caption{Probability distribution function of $X_t$ 
 ($n=10$, $\la_n=10$, $N=200$).}
 \label{fig:1}
\end{figure}

Unit circles of the expected norm restricted
to $\R^2$ can be seen in Figure \ref{fig:2}.
Two different marginal behaviour can be recognized:
If $\la_n=0$, then the expected norm coincides
with the maximum norm. If $\la_n\to\infty$, then 
the expected norm tends to the usual $1$-norm.
\begin{figure}[ht!]
 \centering
 \includegraphics[width=0.6\textwidth]{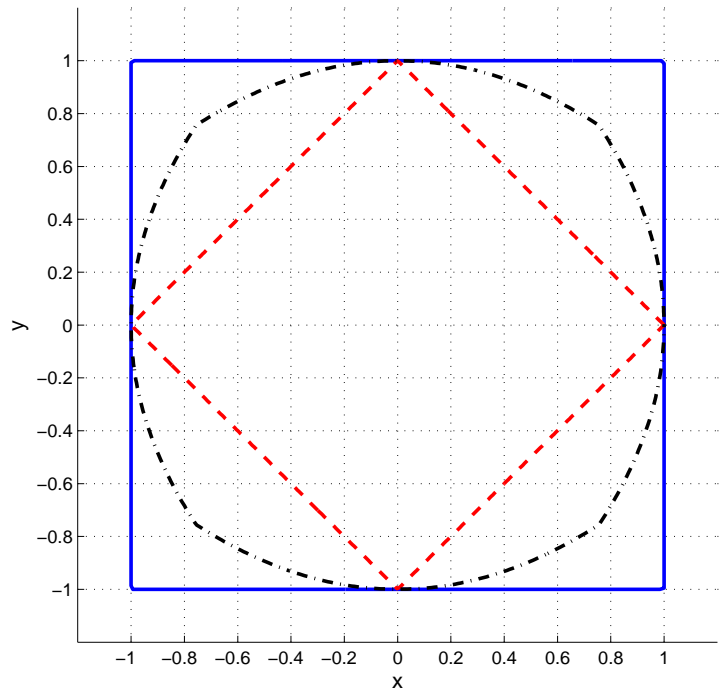}
 \caption{Unit circles of expected norms.
 Solid    -- ($n=1$, $\la_n = 0$, $N=4000$), 
 dashed   -- ($n=1$, $\la_n = 100$, $N=500$),
 dash-dot -- ($n=100$, $\la_n = 100$, $N=1000$).}
 \label{fig:2}
\end{figure}

% Kapott kar. fv. viselkedése midőn n --> végtelen + valszám interpretáció

% Itt kiszámoljuk egy vektor normáját és rajzoltatunk numerikusan egy egységgömböt is
% leírjuk, hogy hogy számoljuk numerikusan a normákat és számolunk ilyet
% n = 1, 10, 100 és végtelen esetére
% A végtelen esetet kiszámoljuk analitikusan

% Expected norm is a norm too 
% Melyik folyamat adja vissza az euklideszi normát?
\bigskip
\section{GSRN in higher dimensions} 

% Eredményeinket általánosítjuk magasabb dimenzióra
% kiszámoljuk numerikusan a 3D egységgömböt adott n-re
% kiszámoljuk analitkusan az egységgömböt n = végtelenre
% bevezetjük a megfelelő l tereket

\bigskip
\subsection{Strong and weak extensions}
%\medskip

One possible way to generalize the results to higher dimensions
is just using the observation that for the familiar $p$-norms
$
||v||_p = ||\left(v_1,||(v_2,\ldots ,v_n)||\right)||_p
$
holds for each $v=(v_1,\ldots ,v_n)\in\C^n$.
We have already constructed GSRN on $\C^2$ ($n=2$).
Suppose that by induction that family of
GSRNs are given on
$\C^k$ where $k<n$. 
Let $v=(v_1,\ldots ,v_n)\in\R^n_{+,\ge}$ be arbitrary and 
$\stackrel{(n-1)}{p}$ is a GSRN defined on 
$\C^{n-1}$. Let $p$ be a GSRN
on $\C^2$ independent from $\stackrel{(n-1)}{p}$
and define
$\stackrel{(n)}{p} (.,v)=p\left(.,\left((v_1,\stackrel{(n-1)}{p}\left(.,(v_2,\ldots ,v_n)\right)\right)\right)$.
Of course, $\stackrel{(n)}{p}$ is a gauge invariant
random norm on $\C^n$, but it is not necessarily
symmetric.
However, the restriction of $\stackrel{(n)}{p}$ 
to $\R^n_{+,\ge}$
defines a GSRN on $\C^n$.
% thus it can be extended through symmetrization to $\C^n$.
If $\stackrel{(n)}{p}$ is a 
GSRN defined on $\C^n$ by using an i.i.d. sequence of
the GSRN $p$ on $\C^2$, then $\stackrel{(n)}{p}$ will be 
called the $n$-dimensional \emph{strong extension}
of $p$. 

The main handicap of the procedure presented above is that for
any $v\in\C^n$ $\stackrel{(n)}{p}$ is a random variable defined on
the $n$ times tensorial product of a probability space which makes
simulations complicated. For this reason we define the \emph{weak extension}
of $p$ which is similar to the strong one just the induction
step is replaced by
$\stackrel{(n)}{p_w} (.,v)=p\left(.,\left((v_1,\E \left(\stackrel{(n-1)}{p_w}\left(.,(v_2,\ldots ,v_n)\right)\right)\right)\right)$.
Unit sphere of the expected norm for $\stackrel{(3)}{p_w}$
is presented in Figure \ref{fig:3}, where 
$\stackrel{(3)}{p_w}$ is the 3-dimensional
weak extension of the GSRN
presented in section \ref{exa}.
\begin{figure}[!ht]
 \centering
 \includegraphics[width=0.8\textwidth]{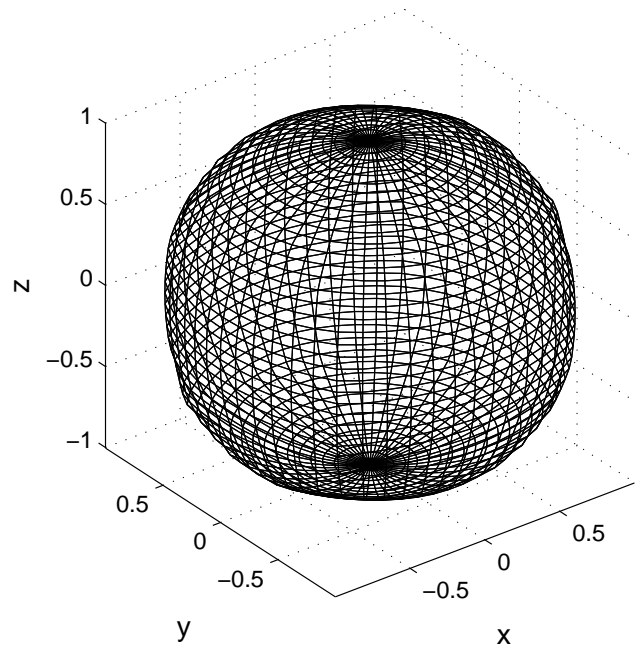}
 \caption{
 Unit sphere of the expected norm defined by the weak extension of the GSRN
 presented in section \ref{exa} ($n=100$, $\la_n=100$, $N=500$).}
 \label{fig:3}
\end{figure}

\bigskip
\subsection{Open problem}
\medskip

Both strong and weak extensions can be generalized to
infinite dimensions as the limit norms of finite dimensional
truncations.
Obviously, property \eqref{inh} is inherited
to infinite dimension hence for any $v=(v_1,v_2,\ldots)$ sequence of complex
numbers which means
$||v||_\infty \le\stackrel{(\infty)}{p}(v) \le ||v||_1$ 
and 
$||v||_\infty \le\stackrel{(\infty)}{p_w}(v)\le ||v||_1$
hold $\PP$-a.s. which implies that
space of sequences for which expected norm is convergent
contains $l^1$ and it is a subspace of $l^\infty$.
This raises many questions. For example:
Let us define an equvalence 
relation between norm processes
in the following way.
Two norm processes are said to be equivalent if
corresponding strong (or weak)
extensions to infinite dimension
define equivalent expected norms.
How can be caracterized equivalence classes of this
equivalence relation?
%\subsection{Construction for infinite dimensional spaces}

%\bigskip
%\section{Numerical simulations} 

%\bigskip
%\section{Applications}
%\medskip

%\newpage

\end{document}